\documentclass[12pt,reqno]{amsart}
\usepackage{amsopn,amssymb,mathrsfs,mathrsfs,a4wide,color,url,microtype}
\usepackage[unicode,linktocpage]{hyperref}
\usepackage{paralist}
\usepackage{fancyvrb}
\usepackage{amscd}

\newcommand{\bib}{\bibitem}

\def\d{\operatorname d}
\def\Z{\mathbb{Z}}

\def\sup{\operatorname{sup}}
\def\Tr{\operatorname{Tr}}

\def\sp{\operatorname{span}}

\def\fii{\varphi}

\def\g{\gamma}

\def\{{\left\lbrace}
\def\}{\right\rbrace}
\def\d{\delta}

\def\im{\operatorname{Im}}

\def\n{\mathcal N}

\def\lip{\operatorname{Lip}}

\def\a{\alpha}

\def\b{\beta}

\def\F{\mathcal F}

\def\L{ L}
\def\e{ \varepsilon}
\def\n{\|}
\def\ln{\left\Vert}
\def\rn{\right\Vert}

\def\({\big(}
\def\){\big)}

\def\({\bigl(}
\def\){\bigr)}

\newtheorem{definition}{Definition}
\newtheorem{theorem}{Theorem}
\newtheorem{lemma}[theorem]{Lemma}
\newtheorem{proposition}[theorem]{Proposition}
\newtheorem{corollary}[theorem]{Corollary}
\theoremstyle{remark}
\newtheorem*{remark*}{Remark}

\newcommand\absb[2]{\csname#1l\endcsname|#2\csname#1r\endcsname|}

\newcommand\normb[2]{\csname#1l\endcsname\|#2\csname#1r\endcsname\|}

\newcommand\N{\mathbb N}
\newcommand\R{\mathbb R}

\allowdisplaybreaks

\begin{document}

\title{Some Remarks on Schauder Bases in Lipschitz Free Spaces}

\author{Mat\v ej Novotn\'y}
\address{Department of Mathematics\\Faculty of Electrical Engineering\\
Czech Technical University in Prague\\ Jugosl\'{a}vsk\'{y}ch Partyz\'an\r{u} 1580/3, 160 00, Prague}

\address{Department of Industrial Informatics\\Czech Institute of Informatics, Robotics, and Cybernetics\\
Czech Technical University in Prague\\ Jugosl\'{a}vsk\'{y}ch Partyz\'an\r{u} 1580/3, 160 00, Prague}
\email{novotny@math.feld.cvut.cz}

\thanks{The work was supported in part by GA\v CR 16-073785, RVO: 67985840 in part by grant SGS18/064/OHK4/1T/13 of CTU in Prague and in part by Ministry of Education, Youth and Sport of the Czech Republic within the project Cluster 4.0 number CZ.02.1.01/0.0/0.0/16\_ 026/0008432.}

\begin{abstract}
We show that the basis constant of every retractional Schauder basis on the Free space of a graph circle increases with the radius. As a consequence, there exists a uniformly discrete subset $M\subseteq\R^2$  such that $\F(M)$ does not have a retractional Schauder basis. Furthermore, we show that for any net $ N\subseteq\R^n$, $n\geq 2$, there is no retractional unconditional basis on the Free space $\F(N)$. 
\end{abstract}

\maketitle

\section{Introduction}
Let $(M,d)$ be a metric space with a distinguished point $0\in M$. Denote $\lip_0(M)$ the space of all Lipschitz functions $f:M\to\R$ with the property $f(0)=0$. Such a space can be equipped with the Lipschitz norm $\n f\n=\sup_{x\neq y}\frac{|f(x)-f(y)|}{d(x,y)}$, which turns it into a Banach space. We see that each point in $M$ can be naturally embedded into $\lip_0(M)^*$ via the Dirac mapping $\d$: $\d_x(f)=f(x)$, $f\in\lip_0(M)$, $x\in M$. The norm-closure of the subspace generated by functionals $\d_x$, $x\in M$, i.e.
$$\overline{\sp}^{\lip_0(M)^*}\{\d_x|x\in M\}$$
is the Lipschitz Free space over $M$, denoted $\F(M)$. Lipschitz Free spaces were introduced already by Arens and Eells in \cite{AE}, although the authors did not use the name Lipschitz Free spaces. Free spaces are called Arens-Eells spaces in \cite{W}, where a lot of results regarding the topic is presented.

Lipschitz Free spaces gained a lot of interest in last decades, connecting nonlinear theory with the linear one. Given two pointed metric spaces $M,N$, every Lipschitz mapping $\fii:M\to N$ which fixes the point $0$ extends to a bounded linear map $F:\F(M)\to\F(N)$, making the following diagram commute:
$$\begin{CD}
\F(M) @>{F}>> \F(N)\\
@A{\d_M}AA @AA{\d_N}A\\
M @>{\fii}>> N
\end{CD}$$
We focus on structural properties of Lipschitz Free spaces. It is well-known that $\lip_0(\R)=\L_\infty$, which yields $\F(\R)=\L_1$ isometrically and similarly $\F(\N)=\ell_1$. In \cite{CDW}, the authors prove that $\F(M)$ contains a complemented copy of $\ell_1(\N)$ if $M$ is infinite (has at least cardinality $\aleph_0$), which was further extended from $\N$ to all cardinalities in \cite{HN}. However, $\F(\R^2)$ cannot be embedded in $\F(\R)=L_1$ (see \cite{NS}).

Certain results were obtained concerning approximation properties in Free spaces, including \cite{PS},\cite{LP},\cite{HLP},\cite{K},\cite{Godefroy},\cite{GO} and of course \cite{GK}. However, not much is known yet about Schauder bases in Free spaces. H\'{a}jek and Perneck\'{a} \cite{HP} constructed a Schauder basis for the Free spaces $\F(\ell_1)$ and $\F(\R^n)$. From \cite{Kaufmann} we have $\F(M)$ is isomorphic to $\F(\R^n)$ for every $M$ with non-empty interior, which gives existence of Schauder basis on such $\F(M)$.

This article follows up the article \cite{HN}, where the authors proved existence (and in the case of $c_0$ constructively) of a Schauder basis on $\F(N)$, for any net $N$ in spaces $C(K)$ for $K$ metrizable compact (hence for $c_0$ and $\R^n$). In section \ref{nonexistence} we show that the same construction as in \cite{HN} cannot be used for constructing bases in $\F(N)$ for arbitrary uniformly discrete subset $N$. In section \ref{unconditionality} we prove that bases constructed in \cite{HN} are not unconditional and that for nets in $\R^n$, no Schauder basis on $\F(N)$ arising from the technique using retractions can be unconditional.
\section{Preliminaries}
As we mentioned, we are interested in constructing a Schauder basis on Lipschitz Free space. However, constructing such basis directly on the Free space is rather complicated, wherefore we prefer to work with its adjoint space and transfer the results to the Free space. The next theorem shows a way to construct a Schauder basis through operators on $\lip_0(M)$.
 
\begin{theorem} \label{operator} Let $M$ be a pointed metric space. Suppose there exists a sequence of linear operators $E_n:\lip_0(M)\to \lip_0(M)$, which satisfies the following conditions:
\begin{enumerate}
\item $\dim E_n \left(\lip_0(M)\right)=n$ for every $n\in\N$,
\item There exists $K>0$ such that $E_n$ is $K$-bounded for every $n\in\N$,
\item $E_m E_n=E_n E_m=E_n$ for every $m,n\in\N$, $n\leq m$,
\item \label{weak}For every $n$, the operator $E_n$ is continuous with respect to topology of pointwise convergence on $\lip_0(M)$,
\item \label{continuity} For every $f\in\lip_0(M)$ the function sequence $E_n f$ converges pointwise to $f$.
\end{enumerate}
Then the space $\F(M)$ has a Schauder basis with the basis constant at most $K$.
\end{theorem}
\begin{proof}
Note first that the topology of pointwise convergence coincides with the $w^*$-topology on bounded subsets of $\lip_0(M)$. Therefore, from the condition (\ref{weak}), the operators $E_n$ are $w^*$ to $w^*$ continuous on bounded subsets of $\lip_0(M)$ and hence there exist linear operators $P_n:\F(M)\to\F(M)$ such that $P^*_n=E_n$ for every $n\in\N$. It is now clear that $\n P_n\n\leq K$, $\dim P_n \left(\F(M)\right)=n$ and that $P_m P_n=P_n P_m=P_n$ for every $m,n\in\N$, $n\leq m$. Furthermore (\ref{continuity}) together with the fact that the topology of pointwise convergence coincides with the $w^*$-topology on bounded subsets of $\lip_0(M)$ means, that for every $f\in \lip_0(M)$ the sequence $E_n f$ converges $w^*$ to $f$, and that for every $\mu\in\F(M)$ the sequence $P_n\mu$ converges weakly to $\mu$. But that means $\n P_n\mu-\mu\n\to 0$ for every $\mu\in\F(M)$. Indeed, if there were $\mu\in\F(M)$, $c>0$ and a subsequence $P_{n_k}$, such that $\n P_{n_k}\mu-\mu\n>c$ for all $k\in\N$, then for every $n\geq n_1$, there exists a $k\in\N$ such that $n\leq n_k$, which yields

$$c<\n P_{n_k}\mu-\mu\n\leq\n P_{n_k}\mu-P_n\mu\n + \n P_{n}\mu-\mu\n\leq (K+1)\n P_n\mu-\mu\n.$$

From $P_1(\F(M))\subseteq P_2(\F(M))\subseteq P_3(\F(M))\subseteq...$ we get $E=\bigcup_{n=1}^{\infty}P_n(\F(M))$ is a convex set and as all $P_n$ are commuting projections, we have that $\mu\notin\overline E$. Indeed, if $\mu \in \overline E$, then there is a sequence $\{x_k\}_{k=1}^\infty\subseteq E$, such that $x_k\to \mu$. If we choose an increasing sequence of numbers $l_k\in\N$, $l_k>n_1$, which satisfy $P_{l_k}x_k=x_k$, we get that
$$\n P_{l_k}x_k-\mu\n\geq \n P_{l_k}\mu-\mu\n-\n P_{l_k}\mu-P_{l_k}x_k\n\geq \frac{c}{K+1}-K\n \mu-x_k\n.$$
Limiting $k\to\infty$ yields $0\geq \frac{c}{K+1}$, which is a contradiction. Therefore $\mu\notin\overline E$. Hence Hahn-Banach theorem gives us the existence of a linear functional $f\in\lip_0(M)$, $\n f\n=1$ with $f|_E=0$ and $f(\mu)>0$. But that is a contradiction as $P_n\mu\overset{w}\to\mu$. Therefore $P_n \mu\to \mu$.
\end{proof}

The following corollary appears already in \cite{HN}, p.12. It gives us a way to construct the Schauder basis on $\F(M)$ only by using the metric space $M$.

\begin{corollary}\label{one} Let $M$ be a metric space with a distinguished point $0$. Suppose there exists a sequence of distinct points $\{\mu_n\}_{n=0}^{\infty}\subseteq M$, $\mu_0=0$, together with a sequence of retractions
$\{\fii_n\}_{n=0}^\infty$, $\fii_n:M\to M$, $n\in\N_0$ which satisfy the following conditions:
\begin{enumerate}[(i)]
\item $\fii_n(M)=\{\mu_j\}_{j=0}^{n}$ for every $n\in\N_0$,\label{bed1}
\item $\overline{\bigcup_{j=0}^{\infty}\{\mu_j\}}=M$, \label{bed2}
\item There exists $K>0$ such that $\fii_n$ is $K$-Lipschitz for every $n\in\N_0$,\label{bed3}
\item $\fii_m\fii_n=\fii_n\fii_m=\fii_n$ for every $m,n\in\N_0$, $n\leq m$. \label{bed4}
\end{enumerate}
Then the space $\F(M)$ has a Schauder basis with the basis constant at most $K$.
\end{corollary}
\begin{proof}
It is not difficult to see that for each $n\in\N$ the formula $E_nf=f\circ\fii_n$, $f\in\lip_0(M)$ defines a linear 
operator $E_n:\lip_0(M)\to\lip_0(M)$, such that the sequence $E_n$ satisfies the assumptions of Theorem \ref{operator}.
\end{proof}

The last two theorems lead us to the following definition.

\begin{definition} Let $M$ be an infinite metric space such that $\F(M)$ has a Schauder basis $E$ with projections $P_n$, $n\in\N$. We say $E$ is an extensional Schauder basis if there exist finite sets $\{0\}=M_0\subseteq M_1\subseteq M_2\subseteq...$ such that $\bigcup_{n=1}^\infty M_n$ is dense in $M$ and we have that for every $n\in\N$ the adjoint $P_n^*$ is a linear extension operator $P_n^*:\lip_0(M_n)\to\lip_0(M)$ with $P_n^*f|_{M_n}=f$ (or equivalently $P_n$ is a projection onto $\F(M_n)$). We say $E$ is a retractional Schauder basis, if there exist retractions $\{\fii_n\}_{n=0}^\infty$, $\fii_n:M\to M$ which satisfy the conditions of Corollary \ref{one} and such that they give rise to the basis $E$, i.e. the adjoints $P_n^*$ satisfy $P_n^*f=f\circ\fii_n$, $f\in\lip_0(M)$.
\end{definition}
It is clear that in the definition we actually have $|M_n\setminus M_{n-1}|=1$ for every $n\in\N$. Note also that every retractional Schauder basis is a special case of an extensional Schauder basis. The next lemma shows in more detail what form the basis vectors take.
\begin{lemma}\label{char}
Let $M$ be a metric space such that there is a sequence of distinct points $0=\mu_0,\mu_1,\mu_2,...\in M$ such that $\bigcup_{n=1}^\infty \{\mu_0,\mu_1,...,\mu_n\}$ is dense in $M$. For every $n\in\N_0$ denote $M_n=\{\mu_0,...,\mu_n\}$. Suppose $\F(M)$ has a Schauder basis $B=\{e_n\}_{n=1}^\infty$. Then the following are equivalent:
\begin{enumerate}
\item\label{extb} $B$ is an extensional Schauder basis with extension operators $E_n:\lip_0(M_n)\to\lip_0(M)$.
\item\label{basf} For every $n\in\N$, there are constants $0\neq c_n,a^n_i\in\R$, $i\in\{1,...,n-1\}$ such that we have $c_n e_n=\d_{\mu_n}-\sum_{i=1}^{n-1}a_i^{n}\d_{\mu_i}$.
\end{enumerate}
\end{lemma}
\begin{proof}
$(\ref{basf})\Rightarrow(\ref{extb})$. Note first that for every $n\in\N$, we have $e_n\in \im P_n\cap\ker P_{n-1}$. From that it follows inductively for every $n\in\N$ that $\im P_n=\sp\{\d_{\mu_1},...,\d_{\mu_n}\}$.

$(\ref{extb})\Rightarrow(\ref{basf})$ The fact that $E_n=P^*_n$ is a bounded linear extension from $M_n$ to $M$ implies that each $P_n$ maps $\F(M)$ onto $\F(M_n)$, which means each basis vector $e_n$ is a linear combination of Dirac functionals at the points of $M_n$, such that the coefficients at $\d_{\mu_n}$ do not vanish.
\end{proof}
Keeping the notation from previous lemma, we see that for each $n\in\N$ we may define a finite dimensional operator $R_{n}:\lip_0(M_{n-1})\to\lip_0(M_{n})$ via
$$R_nf(\mu_{j})=\begin{cases}
f(\mu_j) & j\in\{0,...,n-1\}\,,\\
\sum_{i=1}^{n-1}a_i^{n}f(\mu_{i}) & j=n\,.\\
\end{cases}$$
The operator $E_n=P^*_n$ can be then reconstructed through a $w^*$-limit of operator composition $\lim_k R_kR_{k-1}...R_{n+1}$. The constants $c_n$ were in the lemma only for scaling of the basis vectors $e_n$.

In case of a retractional basis, the basis vectors take form of two-point molecules: For every $n\in\N$ and $i\in\{1,...,n-1\}$ exactly one of the coefficients $a_i^n$ is non-zero, namely has the value $1$. If for example $a_j^n=1$, then $\fii_j(\mu_n)=\mu_j$, which means $e_n=\d_{\mu_n}-\d_{\mu_{j}}$.

Throughout this article, given a metric space $M$, $d$ will denote its metric. If $M$ is a countable (even finite) uniformly discrete metric space with $\F(M)$ having a retractional Schauder basis, by symbols $\mu_0,\mu_1,\mu_2...$, resp. $\fii_0,\fii_1,\fii_2,...$ we will always mean points $\mu_i\in M$, resp. retractions $\fii_i:M\to M$ which satisfy Corollary \ref{one}. Obviously the finite analogues of Corollary \ref{one} and Theorem \ref{operator} also hold. We are going to look in more detail on some properties of retractional Schauder basis.

Following the notation of Corollary \ref{one} (or the proof of Lemma 14 in \cite{HN}) we find useful to denote the set-valued functions $F_i=\fii_i^{-1}:M\to 2^M$, $F_i(x)=\{y|\ \fii_i(y)=x\}$, $i\in\N_0$. Clearly $F_0(0)=M$. From the commutativity of the $\fii_i$'s further follows that for any $i<j$ we have
$$F_i(\mu_i)\cap F_j(\mu_j)\in\{\emptyset,F_j(\mu_j)\}.$$
\begin{definition} Let $M$ and $\mu_i$, $\fii_i$, $i\in\N_0$ satisfy the assumptions from Corollary \ref{one} and $M=\{\mu_i\}_{i=0}^\infty$. A finite or infinite sequence of points $(\mu_{k_1},\mu_{k_2},\mu_{k_3},...)$ is called a chain whenever $k_1<k_2<k_3<...$ and $\fii_{k_i-1}(\mu_{k_i})=\mu_{k_{i-1}}$ for every $i\in\{2,3,...\}$.
\end{definition}
Note that for every chain $(\mu_{k_1},\mu_{k_2},\mu_{k_3}...)$ we have $F_{k_1}(\mu_{k_1})\supseteq F_{k_2}(\mu_{k_2})\supseteq F_{k_3}(\mu_{k_3})\supseteq ...$. We can also introduce partial order on $M$ by $\mu_i\prec\mu_j$ if and only if there exist $n\in\N_0$ points $\mu_{k_1},...,\mu_{k_n}\in M$ such that $(\mu_i,\mu_{k_1},...,\mu_{k_n},\mu_j)$ is a chain. Note also that for two chains $S,T$ the difference $S\setminus T$ and intersection $S\cap T$ are also chains, if nonempty. For a finite chain $(x_1,x_2,...,x_n)$ we call the point $x_1$ its initial point and $x_n$ its final point.

Every chain can be viewed as a path or its segment from $0\in M$ to a given point $x\in M$. Indeed, for every $x\in M$ there exists $n\in\N$ such that for every $i\geq n$ one has $\fii_i(x)=x$. Assuming $n$ is the least number with that property we can define the set $T^{x}_0=\bigcup_{i=0}^{n}\{\fii_i(x)\}$ which contains exactly the points of the chain with initial point $0$ and final point $x$. Regarding $T_0^x$ as an ordered set (the order $\prec$ is linear on $T_0^x$), it is clear that given $x\in M$, there exists exactly one chain $T_0^x$ from $0$ to $x$.

Note also that for every chain $(\mu_{n_1},...,\mu_{n_k})$, $k\geq 2$ there exist constants $c_{n_i}$ (the constants from Lemma \ref{char}) such that for basis vectors $e_{n_1},...,e_{n_k}$ we have
$$\sum_{i=2}^kc_{n_i}e_{n_i}=\d_{\mu_{n_k}}-\d_{\mu_{n_1}}.$$
The following lemma says that, if the space is not too "porous", basis vectors can be made only of two-point molecules in points which are not too far from each other.

\begin{lemma}[Step lemma]\label{step} Let $M$ be a countable metric space, $\a>0, K\geq 1$ and $\fii_n:M\to M$ a system of retractions from Corollary \ref{one}. If $(\mu_{i_1},...,\mu_{i_j})$, $j>1$ is a chain and there exist distinct points $x_1,...,x_k\in M$ with $d(x_l,x_{l+1})\leq\a$, $l\in\{1,...,k-1\}$, $x_1=\mu_{i_j}$, $x_k=\mu_{i_1}$ and $\sup_{i_1\leq n\leq i_j}\lip\fii_n\leq K$, then $d(\mu_{i_{m-1}},\mu_{i_{m}})\leq 2K\a$ for all $m\in\{2,...,j\}$.
\end{lemma}
\begin{proof}
Suppose $d(\mu_{i_{m-1}},\mu_{i_{m}})> 2K\a$ for some $m\in\{2,...,j\}$. We know $\fii_{i_m}(\mu_{i_j})=\mu_{i_m}$ and $\fii_{i_m-1}(\mu_{i_j})=\mu_{i_{m-1}}$. We prove by induction for all $l\in\{1,...,k\}$ that $\fii_{i_m}(x_l)=\mu_{i_m}$ and $\fii_{i_m-1}(x_l)=\mu_{i_{m-1}}$, which is a contradiction as $x_k=\mu_{i_1}$ and $\fii_{i_m}(\mu_{i_1})=\mu_{i_1}\neq \mu_{i_m}$. For $l=1$ we have $x_l=\mu_{i_j}$ and the statement clearly holds. Suppose it holds for all $l=1,...,s-1<k$. From $d(x_{s-1},x_{s})\leq\a$ it follows that $d(\mu_{i_m},\fii_{i_m}(x_{s}))\leq K\a$ and $d(\mu_{i_{m-1}},\fii_{i_{m}-1}(x_{s}))\leq K\a$. From commutativity of all $\fii_n$'s follows that either $\fii_{i_{m}-1}(x_s)=\fii_{i_m}(x_s)\notin\{\mu_{i_m}\}$ holds or $\fii_{i_m}(x_s)=\mu_{i_m}$ and $\fii_{i_m-1}(x_s)=\mu_{i_{m-1}}$ is true. Since $B_{K\a}(\mu_{i_m})\cap B_{K\a}(\mu_{i_{m-1}})=\emptyset$ we conclude the latter is true, which completes the induction step and the contradiction is obtained.
\end{proof}

In the following section, we are going to prove that there are spaces $\F(M)$ with no retractional Schauder basis yet having Schauder basis, moreover extensional.

\section{Nonexistence of retractional Schauder bases}\label{nonexistence}

\begin{definition} Let $x_0,x_1,...,x_n$, $n\in\N$ be distinct points. The set $C^0_n=\{x_0,x_1,x_2...,x_n\}$ with the (standard graph) metric $d(x_k,x_0)=n$, $k\neq 0$, $d(x_k,x_l)=\min\{|k-l|,n-|k-l|\}$, $k,l>0$ we call a circle or a circle of radius $n$ with centre $x_0$.
\end{definition}
In the following, we regard the centre $x_0$ as the base point in the pointed metric space $(C^0_n,d,x_0)$ and denote it $0$.

We are also going to use an uncentered circle, i.e. a subgraph $C_n=\{x_1,x_2,...,x_n\}$ with the induced metric. On $C_n$, we define orientation: We say point $x_l$ lies to the left of the point $x_k$, $k,l\in\{1,...,n\}$, if one of these situations happens:
\begin{enumerate}
\item $k>\frac{n-1}{2}$ and $l\in\{k,k-1,...,k-\lfloor\frac{n+1}{2}\rfloor+1\}$,
\item $k\leq \frac{n-1}{2}$ and $l\in\{k,k-1,...,1\}\cup\{n,n-1,...,n-\lfloor\frac{n+1}{2}\rfloor+k+1\}$.
\end{enumerate}
Analogously, we say $x_l$ lies to the right of $x_k$ if one of the following conditions is satisfied:
\begin{enumerate}
\item $k\leq\frac{n-1}{2}$ and $l\in\{k,k+1,...,k+\lfloor\frac{n+1}{2}\rfloor\}$,
\item $k>\frac{n-1}{2}$ and $l\in\{k,k+1,...,n\}\cup\{1,2,...,\lfloor\frac{n+1}{2}\rfloor-(n-k+1)\}$.
\end{enumerate}
We show that every retractional Schauder basis on $\F(C_n^0)$ has a basis constant which is increasing with $n$.

\begin{theorem}\label{circle}
Let $n\in\N$, $n\geq 10$ and let $\{\fii_i\}_{i=0}^n$ be a system of retractions on a circle $C^0_n$ satisfying the conditions of Corollary \ref{one} . Then there is an $s\in\{1,...,n\}$ such that $$\lip\fii_s\geq\frac{\sqrt{8n+1}-1}{8}.$$
\end{theorem}
\begin{proof}
Let us fix $n\geq 10$ and denote $K=\frac{\sqrt{8n+1}-1}{8}$. We have $\mu_0=0$ and $\mu_1\in C_n$ with $\fii_1(x)=\mu_1$ for all $x\in C_n$ and $\fii_1(0)=0$. Indeed, if $\fii_1(y)=0$ for some $y\in C_n$, then the sets $F_1(0)$ and $F_1(\mu_1)$ have distance $1$. Since they are finite, there exist $w\in F_1(0)$, $z\in F_1(\mu_1)$ such that $d(w,z)=1$ and clearly $d(\fii_1(w),\fii_1(z))=n$, which trivially yields the result, as $n>K$. We prove the theorem by contradiction and assume therefore, $\lip\fii_i< K$ for all $i\in\{1,...,n\}$.

For every point $x\in C_n$ there exists a $k\in\{1,...,n\}$ such that $\{x\}=\{\mu_k\}=\fii_k(C_n)\setminus\fii_{k-1}(C_n)$ and therefore there exists exactly one chain $S_x=(\mu_{k_1},\mu_{k_2},...,\mu_{k_l})$, such that $\mu_{k_1}=\mu_1$ and $\mu_{k_l}=x$ (equivalently $k_1=1$ and $k_l=k$).

Let us introduce sets 
$$A=\{y|\ y\in C_n\setminus\{\mu_1\},d(y,\mu_1)\leq 3K,\text{$y$ lies to the left of $\mu_1$}\}$$
$$B=\{y|\ y\in C_n\setminus\{\mu_1\},d(y,\mu_1)\leq 3K,\text{$y$ lies to the right of $\mu_1$}\}$$
and a mapping $f:C_n\setminus\{\mu_1\}\to\{A,B\}$,
$$f(w)=\begin{cases}
A & \text{there is a $z\in S_w\cap A$ such that for every $y\in S_w$, $z\prec y$, we have $y\notin A\cup B$}\,, \\ 
B & \text{there is a $z\in S_w\cap B$ such that for every $y\in S_w$, $z\prec y$, we have $y\notin A\cup B$}\,. \\
\end{cases}$$
Note that the definitions of $A,B$ make perfect sense, as $3K<\frac{n}{2}$. Also, the mapping $f$ is well-defined, as for every $w\in C_n\setminus\{\mu_1\}$ the intersection $S_w\cap(A\cup B)$ is nonempty. Indeed, according to Step lemma \ref{step} applied on the $C_n$, the distance between any two adjacent points in a chain is smaller than $2K$ and therefore for the second element $z\in S_w$ (meaning $S_w=(\mu_1,z,...,w)$) we have $d(\mu_1,z)\leq 2K$ and thus $z\in A$ or $z\in B$.

Observe that $f(w)=A$ for every $w\in A$ and $f(w)=B$ for every $w\in B$. We prove there exist two points $a,b\in C_n\setminus\left(\{\mu_1\}\cup A\cup B\right)$ such that $d(a,b)=1$, $f(a)=A$ and $f(b)=B$.

Let us assume for contradiction that $f(w)=A$ for all points $w\in C_n\setminus\left(\{\mu_1\}\cup A\cup B\right)$. Denote $z$ the closest point to the right of the set $B$, i.e. the only point with $3K <d(z,\mu_1)\leq 3K+1$ and $d(z,B)=1$. We have $f(z)=A$, which means the chain $S_z=(\mu_1,\mu_{k_2},...,z)$ leaves the set $A$ and goes to the left around (meaning omitting the set $A\cup B$) the circle to the point $z$, with steps smaller than $2K$. Therefore there exists a point $\mu_l\in S_z$ such that $d(\mu_1,\mu_l)\geq\frac{n-2K}{2}$. But then we have $\fii_{l}(\mu_1)=\mu_1$, $\fii_{l}(z)=\mu_l$, which yields
$$\lip\fii_l\geq \frac{d(\mu_1,\mu_l)}{d(\mu_1,z)}\geq\frac{n-2K}{2(3K+1)}\geq K,$$
as $n\geq 10$, which contradicts our assumption.

Therefore, let there exist two points $a,b\in C_n\setminus\left(\{\mu_1\}\cup A\cup B\right)$ such that $d(a,b)=1$, $f(a)=A$ and $f(b)=B$. Consider the two chains $S_a=(\mu_1,\mu_{k_1},\mu_{k_2}...,a)$ and $S_b=(\mu_1,\mu_{l_1},\mu_{l_2}...,b)$ and let $i$ and $j$ be such that $\mu_{k_i}\in A$, $\mu_{l_j}\in B$ and we have $\mu,\nu\notin A\cup B\cup\{\mu_1\}$ for every $\mu\in S_a$, $\mu_{k_i}\prec \mu$ and every $\nu\in S_b$, $\mu_{l_j}\prec \nu$. Note that $d(\mu_1,\mu_{l_j})\geq K+1$ and $d(\mu_1,\mu_{k_i})\geq K+1$.

Without loss of generality suppose $k_i<l_j$. Then $\fii_{l_j}(b)=\mu_{l_j}$ and $d(\fii_{l_j}(a),\mu_{l_j})\leq K$. This implies $u_a:=\fii_{l_j}(a)$ has distance at most $K$ from the set $B$, which yields $d(\mu_1,u_a)\leq 4K$ and the chain $S=(\mu_{k_i},...,u_a)$ must go from the set $A$ to the left around the circle closer to the set $B$. Thus there must exist a point $v=\mu_s\in S$ such that $d(v,\mu_1)\geq\frac{n-2K}{2}$. It follows that $$\lip \fii_s\geq \frac{d(\fii_s(u_a),\fii_s(\mu_1))}{d(u_a,\mu_1)}=\frac{d(v,\mu_1)}{d(u_a,\mu_1)}\geq\frac{n-2K}{8K}=K,$$ which is again a contradiction. We conclude there exists an $s\in\{1,...,n\}$ such that $\lip\fii_s\geq K=\frac{\sqrt{8n+1}-1}{8}$.
\end{proof}
\begin{corollary}\label{no basis} There exists a uniformly discrete set $N\subseteq\R^2$ such that the Free space $\F(N)$ has no retractional Schauder basis.
\end{corollary}
\begin{proof}
Let $N=\bigcup_{n=1}^\infty C^0_{4^n}$ be a union of circles with the same centre $0$ and with radii $4^n$, $n\in\N$. Suppose $d_n$ is the metric on $C^0_{4^n}$. Let us define a metric on $N$ in the following way: 
$$d(x,y)=\begin{cases}
d_n(x,y) & \text{ if } x,y\in C^0_{4^n}\,,\\
\max\{4^i,4^j\} & \text{ if } x\in C_{4^i},y\in C_{4^j},i\neq j\,.\\
\end{cases}$$
It is clear that $d$ is indeed a metric on $N$ and one has no difficulties to embed $N$ into $\R^2$ in a bilipschitz way, actually with distortion not worse than $2\pi$.
We show that every sequence of retractions $\fii_i:N\to N$ satisfying conditions $(i)$ and $(iv)$ from Corollary \ref{one} cannot satisfy the condition $(iii)$ of that corollary.

Let therefore $\fii_i:N\to N$ be a commuting sequence of retractions such that $\fii_0(0)=0$ and $|\fii_i(N)|=i+1$. We show that for every $k\in\N$, $k\geq 4$, there exists an $n=n_k\in\N$ such that $\lip\fii_{n_k}\geq k.$ Pick therefore $k\in\N$, $k\geq 4$, and find the smallest $n$ such that $\mu_n\in C_{4^k}$. Then $\fii_i(\mu_n)=\mu_n$ for every $i\geq n$. If there exist $j\geq n$ and $x\in C_{4^k}$ such that $\fii_j(x)\notin C_{4^k}$, we have $\lip\fii_j\geq 4^k\geq k$ and the proof is finished. Indeed, if we take $x\in C_{4^k}$ such that $\fii_j(x)\notin C_{4^k}$ and without loss of generality we assume $x$ is such that $d(x,\mu_n)$ is minimal among all $x\in C_{4^k}$ with $\fii_j(x)\notin C_{4^k}$, we have $d\left(\fii_j(y),\fii_j(x)\right)\geq 4^k\geq k$ for one of $x$'s neighbours $y$ (i.e. $d(x,y)=1$). This means $\lip\fii_j\geq k$.
If, on the contrary, we have $\fii_i(x)\in C_{4^k}$ for all $i\geq n$ and all $x\in C_{4^k}$, we find ourselves in the case of Theorem \ref{circle}. Indeed, if we view the circle $C^0_{4^k}$ as a set $C^0_{4^k}=\{0,\mu_{s_1},\mu_{s_2},...,\mu_{s_{4^k}}\}$ (for some eligible $s_1,s_2,...,s_{4^k}\in\N$) and look only at retractions $\fii_0,\fii_{s_1},\fii_{s_2},...,\fii_{s_{4^k}}$ restricted to the circle $C^0_{4^k}$, we apply \ref{circle} and conclude $\max\{\lip\fii_{s_1},\lip\fii_{s_2},...,\lip\fii_{s_{4^k}}\}\geq\frac{\sqrt{8\cdot 4^k+1}-1}{8}\geq k$.
\end{proof}
We see it is impossible to build a retractional Schauder basis on $\F(N)$. However, the space $\F(N)$ has an extensional Schauder basis as we are going to show in the next proposition:
\begin{proposition} Let $N=\bigcup_{n=1}^\infty C^0_{4^n}$ be the metric space from Corollary \ref{no basis}. Then $\F(N)$ has an extensional monotone Schauder basis.
\end{proposition}
\begin{proof} First, note that we have orientation of every $C_{4^n}$, $n\in\N$. For every $i\in\N$, define $k=k(i)$ as the unique integer such that $\frac{4^{k}-1}{3}\leq i< \frac{4^{k+1}-1}{3}$. Let $N=\{0,x_1,x_2,x_3,...\}$ be enumerated in such way that for every $i\in\N$ we have $x_i\in C_{4^{k}}$ and that the enumeration respects the orientation on every circle $C_{4^k}$. Namely, if $x_i,x_{i+1}\in C_{4^k}$, we have that $d(x_i,x_{i+1})=1$ and $x_{i+1}$ lies to the right of $x_i$. Denote $D_i=\{0,x_1,x_2,...x_{i}\}$.

We are going to define a sequence of extension operators $P_i:\lip_0(D_i)\to\lip_0(N)$ and prove they satisfy the assumptions of Theorem \ref{operator}. In order to do that, let us define some preparatory notions. Define the left and the right "$D_i$-neighbour" functions $\nu^l_i,\nu^r_i:\bigcup_{n=1}^{k(i)} C_{4^n}^0\to D_i$ as follows: For each $n\in\{1,2,...,k(i)\}$ and $x\in C_{4^{n}}$, let $\nu^l_i(x)\in D_i$ be the closest point to the left of $x$ and let $\nu^r_i(x)\in D_i$ be the closest point to the right of $x$. We set $\nu^l_i(0)=\nu^r_i(0)=0$. Note that $\nu^l_i(x)=\nu^{r}_i(x)=x$ if and only if $x\in D_i$. Further we need to define "right-" and "left-" metric function (not proper metrics) on every circle $C_{4^n}$. For points $x,y\in C_{4^n}$ we set the value $d^l(x,y)$ as the length of the path (in the graph $C_{4^n}$) going from $x$ to the left up to $y$. Analogously, we set $d^r(x,y)$ as the length of the path going from $x$ to the right up to $y$. It is clear that for $x,y\in C_{4^n}$ we have $d^l(x,y)=d^r(y,x)$ and $d(x,y)=\min\{d^r(x,y),d^l(x,y)\}$.

Further we define for every $i\in\N$ the $i$-th interpolation function $I_i:\lip_0(D_i)\times \bigcup_{n=1}^{k(i)} C_{4^n}^0\to\R$ via
$$I_i(f,x)=\frac{d^r(x,\nu^r_i(x))f(\nu^l_i(x))+d^l(x,\nu^l_i(x))f(\nu^r_i(x))}{d^l(x,\nu^l_i(x))+d^r(x,\nu^r_i(x))}\ \ \  \text{if } x\neq \nu^l_i(x) \text{ or } x\neq \nu^r_i(x)$$
and $I_i(f,x)=f(x)$ for $x=\nu^l_i(x)=\nu^r_i(x)$. Clearly, $I_i(f,x)$ is the value of linear interpolation of the function $f$ between closest points of $x$ to the left and to the right from the set $D_i$, given we take $x$ itself to be the closest point to $x$ in any direction if $x\in D_i$. Let now $f\in\lip_0(D_i)$. Then we define our (extension) operators $P_i$, $i\in\N$ simply as
$$P_if(x)=\begin{cases}
I_i(f,x) & x\in\bigcup_{n=1}^{k(i)} C_{4^n}^0\,,\\
0 & x\in\bigcup_{n=k(i)+1}^{\infty} C_{4^n}
\end{cases}$$
and of course, $P_0=0$. Clearly $P_i$ is a linear operator for every $i\in\N$ and the function $P_if$ is Lipschitz with the same constant as $f$. Indeed, if we take $x\in C_{4^{n}}$ and $y\in C_{4^{m}}$ with $m<	n$, we see from the definition of $I_i$ that $\min_{z\in C_{4^{n}}}f(z)\leq P_if(x)\leq\max_{z\in C_{4^{n}}}f(z)$ and $\min_{w\in C_{4^{m}}}f(w)\leq P_if(y)\leq\max_{w\in C_{4^{m}}}f(w)$. From that and from the fact that $d(z,w)=d(x,y)=4^n$ holds for all $z\in C_{4^n}$ and $w\in C_{4^m}$, we get
$$|P_if(x)-P_if(y)|\leq \max_{\substack{
            z\in C_{4^{n}}\\
            w\in C_{4^{m}}}}
|f(z)-f(w)|\leq 4^n \n f\n=d(x,y)\n f\n.$$
For $x\in C_{4^{n}}$ and $0$ we have clearly $|P_if(x)-P_if(0)|\leq\max_{z\in C_{4^{n}}}|f(z)|\leq d(x,0)\n f\n$.

The only nontrivial case to prove is the case $x,y\in C_{4^{k(i)}}$. Let therefore $x,y\in C_{4^{k(i)}}$. There are three cases. If $x,y\in D_i$, then $P_if(x)=f(x)$ and $P_if(y)=f(y)$, which is trivial. Let $x,y\notin D_i$ and $\nu^r_i(x)=\nu^r_i(y)=a$, $\nu_i^l(x)=\nu_i^l(y)=b$. We can assume $d^r(b,x)\leq d^r(b,y)$, for the roles of $x$ and $y$ are symetrical. From that we have $d^r(y,a)\leq d^r(x,a)$. If $d(x,y)=d^r(x,y)$, we have
\begin{align*}
|P_if(x)-P_if(y)|&=\left|\frac{f(b)d^r(x,a)+f(a)d^r(b,x)}{d^r(b,a)}-\frac{f(b)d^r(y,a)+f(a)d^r(b,y)}{d^r(b,a)}\right|\\
&=\left|\frac{f(b)d^r(x,y)-f(a)d^r(x,y)}{d^r(b,a)}\right|\\
&\leq \frac{\n f\n d(a,b)}{d^r(b,a)}d^r(x,y)\leq \n f\n d(x,y).
\end{align*}
If $d(x,y)=d^l(x,y)$, then $d(x,y)=d^r(y,a)+d^r(a,b)+d^r(b,x)$ and then from $d^r(b,a)-d^r(x,y)=d^r(b,x)+d^r(y,a)$ we have by triangle inequality
\begin{align*}
|P_if(x)-P_if(y)|&=\left|\frac{f(b)d^r(x,y)-f(a)d^r(x,y)}{d^r(b,a)}\right|\\
&=\left|\frac{f(b)\left(d^r(x,y)-d^r(b,a)\right)+(f(b)-f(a))d^r(b,a)+f(a)\left(d^r(b,a)-d^r(x,y)\right)}{d^r(b,a)}\right|\\
&=\left|\frac{\left(d^r(b,x)+d^r(y,a)\right)\left(f(a)-f(b)\right)+(f(b)-f(a))d^r(b,a)}{d^r(b,a)}\right|\\
&\leq \n f\n\left(\frac{d(a,b)}{d^r(a,b)}\left(d^r(b,x)+d^r(y,a)\right)+d(a,b)\right)\leq\n f\n d(x,y)\\
\end{align*}
The case $x\in D_i$, $y\notin D_i$ is proved in a similar way.

We see that the functions $f_j=P_i\left(\chi_{\{x_j\}}\right)$, $1\leq j\leq i$ create a basis of each space $P_i(\lip_0(N))$, hence $\dim P_i(\lip_0(N))=i$ for every $i\in\N$.

To prove the commutativity it suffices to prove $P_{i+1}P_i=P_iP_{i+1}=P_i$ for every $i\in\N$. While $P_{i}P_{i+1}=P_i$ is clear, we prove for every $f\in\lip_0(N)$ we have $P_{i+1}P_if=P_if$. Fix $f\in\lip_0(N)$. If $k(i+1)>k(i)$, then $D_i=\bigcup_{n=1}^{k(i)} C_{4^n}^0$ and $P_{i+1}P_i f(x)=f(x)=P_i f(x)$ for all $x\in \bigcup_{n=1}^{k(i)} C_{4^n}^0$ and $P_{i+1}P_i f(x)=0=P_i f(x)$ for all $x\notin \bigcup_{n=1}^{k(i)} C_{4^n}^0$. Let therefore $k(i+1)=k(i)$.

Denote $a=x_i=\nu^l_{i}(x_{i+1})$ and $b=\nu^r_{i}(x_{i+1})$. All we need to check is $P_{i+1}P_i f(y)=P_i f(y)$ holds for all $y\in C_{4^{k(i)}}\setminus D_i$. Indeed, for all other points $x$ we have $P_if(x)=P_{i+1}f(x)$. Take therefore a point $y\neq x_{i+1}$ (otherwise it is trivial). Note that $\nu^l_{i}(y)=a$, $\nu^r_{i}(y)=b$ and that $d^r(a,x_{i+1})=1$. Then we have
\begin{align*}
P_{i+1}(P_i f)(y)&=\frac{d^r(x_{i+1},y)P_if(b)+d^r(y,b)P_if(x_{i+1})}{d^r(x_{i+1},b)}\\
&=\frac{d^r(x_{i+1},y)f(b)+d^r(y,b)\cdot\frac{f(b)+d^r(x_{i+1},b)f(a)}{d^r(a,b)}}{d^r(x_{i+1},b)}\\
&=\frac{d^r(x_{i+1},y)d^r(a,b)+d^r(y,b)}{d^r(x_{i+1},b)d^r(a,b)}\cdot f(b)+\frac{d^r(y,b)}{d^r(a,b)}\cdot f(a)\\
&=\frac{d^r(x_{i+1},y)d^r(x_{i+1},b)+d^r(x_{i+1},y)+d^r(y,b)}{d^r(x_{i+1},b)d^r(a,b)}\cdot f(b)+\frac{d^r(y,b)}{d^r(a,b)}\cdot f(a)\\
&=\frac{d^r(x_{i+1},b)\left(1+d^r(x_{i+1},y)\right)}{d^r(x_{i+1},b)d^r(a,b)}\cdot f(b)+\frac{d^r(a,y)}{d^r(a,b)}\cdot f(a)\\
&=\frac{d^r(a,y)}{d^r(a,b)}\cdot f(b)+\frac{d^r(a,y)}{d^r(a,b)}\cdot f(a)\\
&=P_i f(y)\\
\end{align*}
and the commutativity is proved.

Let $i\in\N$. If $f_{\a}\to f$ pointwise, then for every $x\in D_i$ we have $P_if_{\a}(x)=f_{\a}(x)\to f(x)=P_if(x)$ and for every $x\in\bigcup_{l=k(i)+1}^\infty C_{4^{l}}$ we have $P_if_{\a}(x)=0=P_if(x)$. Finally, for every $x\in C_{4^{k(i)}}\setminus D_i$ we have $P_if_{\a}(x)=\g_x f_{\a}(a_x)+(1-\g_x)f_{\a}(b_x)$, for some eligible $\g_x\in [0,1]$, $a_x,b_x\in D_i$ and the choice of these points depends only on $x$ (and $i$ of course). Therefore $P_if_{\a}\to P_if$ pointwise, which means that every operator $P_i$ is continuous with respect to topology of pointwise convergence. 

Finally the sequence $P_if$ converges pointwise to $f$. Indeed, for every $y\in N$ there exists $i\in\N$ such that $y\in D_i\subseteq D_{i+1}\subseteq D_{i+2}\dots$, which yields $P_i f(y)=P_{j}f(y)=f(y)$ for all $j\geq i$. Hence $P_i f\to f$ pointwise.

Since the operators $P_i$ meet all assumptions from Theorem \ref{operator}, we get that there is a sequence of operators $T_i:\F(N)\to\F(N)$, $i\in\N_0$ with $T_i^*=P_i$ which build a monotone Schauder basis for $\F(N)$.
\end{proof}
\begin{remark*} It was not necessary for the construction of $P_i$'s to enumerate the set $N$ with respect to orientation on every circle $C_{4^k}$. Actually any enumeration which satisfies $x_i\in C_{4^{k}}$ for every $i\in\N$ works. Our choice only slightly simplifies the proof.
\end{remark*}
\section{Unconditionality of retractional Schauder bases}\label{unconditionality}
As we construct a Schauder basis on $\F(M)$ via sequence of retractions, as described in Corollary \ref{one}, properties of such a basis depend also on properties of the metric space $M$. Naturally it leads us to the question: What can $M$ be like such that there is an unconditional retractional Schauder basis on $\F(M)$? The next lemma sets a condition on the chains under which the acquired basis is conditional. It is further used in Theorem \ref{main}, which shows that retractional bases on Free spaces of nets in finite-dimensional spaces are conditional.
\begin{lemma}\label{alligned chains} Let $\a,\b>0$ and let $N$ be an $\a$-separated metric space, such that there exist retractions $\fii_i:N\to N$ satisfying the conditions from Corollary \ref{one}. Suppose there exists $n_0\in\N$ such that for every $n\in\N$, $n\geq n_0$ there exist chains $S=(\mu_{0},\mu_{k_1},...,\mu_{k_s})$ and $T=(\mu_0,\mu_{l_1},...,\mu_{l_m})$, $s,m\in\N$ with $d(\mu_{k_s},\mu_{l_m})\leq\b$ and $|S\setminus T|\geq n$. Then the retractional Schauder basis on $\F(N)$ corresponding to the retractions $\fii_i$ is conditional.
\end{lemma}
\begin{proof}
Let now $P_i$ be the associated Schauder projection to the mapping $\fii_i$ for each $i\in\N_0$, i.e. the projection to the subspace $\sp\{\d_{\mu_0},\d_{\mu_1},...,\d_{\mu_i}\}$. Instead of working directly with $P_0,P_1,P_2,...$ we will use their adjoints $P_0^*,P_1^*,P_2^*,...$ and for every $n\in\N$, $n\geq n_0$ we construct a function $f_n\in\lip_0(N)$ with $\n f_n\n\leq 1$ and find a sequence of signs $\e_0,\e_1,...,\e_{k_s}$ for some $s\geq n$ such that the following inequality holds 
$$\ln\sum_{i=0}^{k_s}\e_i(P_{i+1}^*-P_{i}^*)f_n\rn\geq \frac{\a (n-1)}{\b}.$$
Fix $n\in\N$ and chains $S=(\mu_0,\mu_{k_1},...,\mu_{k_s})$, $T=(\mu_0,\mu_{l_1},...,\mu_{l_m})$ for which we have $d(\mu_{k_s},\mu_{l_m})<\b$ and $|S\setminus T|\geq n$. Suppose now $t\in\{0,1,2,...,s-n\}$ is such that $\mu_{k_t}\in T$ and $\mu_{k_{t+1}}\notin T$ (we set $\mu_{k_0}=\mu_0$). We define the function $f_n$ on $N$ via the formula

$$f_n(x)=\begin{cases}
\frac{\a}{2} & x=\mu_{k_j}\text{ for } j \text{ odd},j>t\,,\\
\frac{-\a}{2} & x=\mu_{k_j}\text{ for } j \text{ even},j> t\,,\\
0 & \text{else}\,.
\end{cases}$$
Clearly, $f_n(\mu_0)=0$ and $\n f_n\n\leq 1$. For the following choice of sings $\e_0=1$,
$$\e_i=\begin{cases}
-\e_{i-1} & i=k_j \text{ for some }j\in\N\,,\\
\e_{i-1} & \text{else}\,,
\end{cases}$$
we have
$$\sum_{i=0}^{k_s}\e_i(P_{i+1}^*-P_{i}^*)=-P_0^{*}+2\sum_{j=1}^s(-1)^{j+1}P^*_{k_j}+(-1)^sP^{*}_{k_s+1}=:P$$
and then
\begin{align*}
\ln P\rn&\geq \ln P f_n\rn\geq \ln \frac{Pf_n(\mu_{k_s})-Pf_n(\mu_{l_m})}{d(\mu_{k_s},\mu_{l_m})}\rn\geq\frac{1}{\b}\ln Pf_n(\mu_{k_s})-Pf_n(\mu_{l_m})\rn=\\
&=\frac{1}{\b}\left\vert -f_n(0)+2\sum_{j=1}^s(-1)^{j+1}f_n(\mu_{k_j})+(-1)^sf_n(\mu_{k_s})+0\right\vert\\
&=\frac{1}{\b}\left\vert 2\sum_{j=t+1}^s\frac{\a}{2}-\frac{\a}{2}\right\vert\geq\frac{\a(s-t-1)}{\b}\geq\frac{\a (n-1)}{\b}.\\
\end{align*}
\end{proof}
Recall that a subset $S$ of a metric space $M$ is called an $\a,\b$-net whenever $S$ is $\a$-separated and $\b$-dense in $M$, i.e. $\inf_{x\neq y}d(x,y)\geq \a$, $x,y\in S$ and $\sup_{x\in M} d(x,S)\leq \b$.

In \cite{HN}, the authors constructed a system of retractions on the integer lattice in $c_0$ which satisfies the conditions of Corollary \ref{one}. Through suitable homomorphisms they further showed the existence of a basis on any Free space of a net in a separable $C(K)$ space or a net in $c_0^+$, the positive cone in $c_0$.
 
\begin{corollary} Let  $N$ be a net in any of the following metric spaces: $C(K)$, $K$ metrizable compact, or
$c_0^+$ (the subset of $c_0$ consisting of elements with non-negative coordinates). The basis on $\F(N)$ constructed in \cite{HN} is conditional.
\end{corollary}
\begin{proof}
First we consider the case $N=\Z^{<\omega}\subseteq c_0$, the integer lattice in $c_0$. Following the proof of Lemma $14$ in \cite{HN} we see, there are chains which go parallelly along the first coordinate axis (or any other coordinate axis). Every such two chains hence satisfy the conditions of the previous lemma, which yields that a basis arising from these retractions cannot be unconditional. As the existence of bases in other cases than $N$ being the integer lattice in $c_0$ was proven only by isomorphisms, we conclude that none of them are unconditional.
\end{proof}
\begin{theorem}\label{main}
Let $N$ be an $\a,\b$-net in a finite-dimensional normed space $X$ with $\dim X\geq 2$. Let $E=\{e_i\}_{i=1}^\infty$ be a retractional Schauder basis on $\F(N)$. Then $E$ is conditional.
\end{theorem}
In the following, $B_{\e}(x)$ denotes closed ball of radius $\e>0$ and centre $x\in X$, $B_{\e}^{\circ}(x)$ denotes its interior. In the same way $B_{\e}:=B_{\e}(0)$ and $S_{\e}$ denotes sphere of radius $\e$ and centre $0$.
\begin{proof}
Let $\fii_i:N\to N$ be the corresponding retractions to the basis $E$. We prove the theorem by showing that the assumptions of Lemma \ref{alligned chains} are met. Denote $\sup_{i\in\N}\lip\fii_i=K<\infty$. Pick $n\in\N$, such that $n>8K$. Define annulus with radii $r$ and $w$, $w<r$ as $A(r,w)=B_{r+w}(0)\setminus B^{\circ}_{r-w}(0)$. Our aim is to prove there exist chains $T,Z$ with final points $t,z\in A(3K\b n+\b,\b)\cap N$ with $d(t,z)\leq 2\b$ such that $x\in B_{K\b n}$ holds for $x=x_{t,z}$, the final point of the chain $T\cap Z$. Then we have $d(t,x)\geq 3K\b n-K\b n=2K\b n$ and Step lemma \ref{step} yields $|T\setminus Z|\geq n$, which by Lemma \ref{alligned chains} concludes the proof.

For the following, for every two points $x,y\in N$ with $x\prec y$ denote $T_x^y$ the chain with initial point $x$ and final point $y$. Assume now for contradiction, for every pair of points $t,z\in A(3K\b n+\b,\b)\cap N=:A$ with $d(t,z)\leq 2\b$ the final point $x_{t,z}$ of $T_0^t\cap T_0^z$ lies outside the ball $B_{K\b n}$. That means there exists a point $\mu_m\in N$ for some $m\in\N$ with $d(0,\mu_m)>K\b n$ such that $\fii_m(t)=\mu_m$ for every $t\in A$. To prove this, note that $0\in\bigcap_{t\in A}T_0^t$ and as $\bigcap_{t\in A}T_0^t$ is a chain, it has a final point which we denote $\mu_m$ and prove that $d(0,\mu_m)>K\b n$. We show that for every two points $t,z\in A$ the final point $x_{t,z}$ of the chain $T_0^t\cap T_0^z$ is of greater norm than $K\b n$. Clearly, if $d(t,z)\leq 2\b$, the statement holds as assumed. If $d(t,z)> 2\b$, we can find a finite sequence of points $y_1,...,y_l\in A$,  $l\in\N$ such that $d(y_i,y_{i+1})\leq 2\b$ for every $i\in\{1,...,l-1\}$ and that $y_1=t$ and $y_l=z$. Then $x_{t,z}\in\{x_{y_i,y_{i+1}}|\ i\in\{1,...,l-1\}\}$, which means $\n x_{t,z}\n> K\b n$. Note that for any three points $s,t,z\in A$ the final point $x_{s,t,z}$ of the chain $T_0^s\cap T_0^t\cap T_0^z$ is equal to one of the points $x_{s,t},x_{t,z},x_{s,z}$. Indeed, as $x_{s,t},x_{t,z}\in T_0^t$, we have that either $x_{s,t}\prec x_{t,z}$ or $x_{s,t}\succ x_{t,z}$. If $x_{s,t}\prec x_{t,z}$, then $x_{s,z}=x_{s,t}=x_{s,t,z}$ and the other case follows symmetrically. But from that we get inductively that for any finite number of points $t_1,...,t_v$, there are indices $i,j\in\{1,...,v\}$, such that the final point $x_{t_1,...,t_v}$ of the chain $\bigcap_{l=1}T_0^{t_l}$ equals $x_{t_i,t_j}$. Because for each two $t,z\in A$ we have $\n x_{t,z}\n>K\b n$ and $A$ is finite, we have $\n \mu_m\n> K\b n$.

Observe further, that $T_{\mu_m}^t\cap B_{\b n}=\emptyset$ holds for every chain $T_{\mu_m}^t$ with initial point $\mu_m$ and final point $t\in A$. Indeed, if $\mu_p\in T_{\mu_m}^t$, $p\in\N$ is such that $\n\mu_p\n\leq \b n$, we have $\lip\fii_m\geq \frac{\n\fii_m(0)-\fii_m(\mu_p)\n}{\n \mu_p\n}=\frac{\n\mu_m\n}{\n\mu_p\n}>\frac{K\b n}{\b n}=K$, which is not possible.

Let us denote $S=\bigcup_{t\in A}T_{\mu_m}^{t}$ the set of all chains from $\mu_m$ to points of $A$. Let $S=\{\mu_{k_1},...,\mu_{k_q}\}$ for some $k_1<k_2<...<k_q$, $q\in\N$. Note that $\mu_{k_1}=\mu_m$. For every chain $T=(t_1,...,t_l)$, $l\in\N$ define a trajectory of the chain $\Tr(T)$ as the union of the line segments $\bigcup_{i=1}^{l-1} [t_i,t_{i+1}]$. Denote $D=\bigcup_{t\in A}\Tr\left(T_{\mu_m}^{t}\right)$. Define now a function $F:[1,q]\times A\to D$ via 
$$F(t,x)=(t-i)\fii_{k_{i+1}}(x)+(1-(t-i))\fii_{k_{i}}(x), x\in A, t\in [i,i+1), i\in\{1,...,q-1\}$$
and $F(q,x)=x$, $x\in A$. We see that for each $t\in [1,q]$, the function $F(t,\cdot)$ is $K$-Lipschitz and that we have $F(1,x)=\mu_m$ for all $x\in A$.

Let $\e=3K\b n+\b$ and consider $B=\{B^{\mathrm{o}}_{2\b}(x)\cap S_{\e}\}_{x\in A}$ as an open cover of $S_{\e}$ and find a partition of unity $\{\psi_a\}_{a\in A}$ subordinated to the cover $B$. Define a function $R:[1,q]\times S_{\e}\to X$ by

$$R(t,x)=\sum_{a\in A}\psi_a(x)F(t,a),\ \ t\in [1,q],\ x\in S_{\e}.$$

We see, that $R(1,x)=\mu_m$ for all $x\in S_{\e}$ and that $$\sup_{x\in S_{\e}}|R(q,x)-x|\leq 2\b.$$ Of course, $R$ is continuous on $[1,q]\times S_{\e}$. Our goal is to prove there is a continuous deformation of  $S_{\e}$ into one point $\mu_m$ avoiding the origin, which is a contradiction. For that we define a straight-line homotopy between identity and $R(q,\cdot)$ by $W:[0,1]\times S_{\e}\to A(\e,2\b)$, $W(t,x)=tR(q,x)+(1-t)x$. Joining mappings $W$ and $R$ we get a mapping $Z:[0,q]\times S_{\e}\to X$ precisely defined by
$$Z(t,x)=\begin{cases}
W(t,x) & t\in[0,1), x\in S_{\e}\,,\\
R\left(\frac{q}{t},x\right) & t\in[1,q], x\in S_{\e}\,.\\
\end{cases}$$
All there is left to prove is that $R([1,q]\times S_{\e})\cap \{0\}=\emptyset$. To see that, note that the value $R(t,x)$ is a convex combination of values $F(t,a)$, where $a\in A$ are such that $d(x,a)<2\b$. Fix therefore $x\in S_{\e}$ and let $\mu_{l_1},\mu_{l_2},...,\mu_{l_p}\in A$ be such that $d(x,\mu_{l_i})<2\b$ for all $i$. From the preceding paragraphs it follows that the trajectory $\Tr(T_{\mu_m}^{\mu_{l_i}})$ of each chain from $\mu_m$ to $\mu_{l_i}$ has no intersection with $B_{\b \frac{n}{2}}$. Indeed, as the chain $T_{\mu_{m}}^{\mu_{l_i}}$ avoids the ball $B_{\b n}$ and the distance between two consecutive points in a chain is bounded by $2\b K$ and $n>2K$, we get the result. From the fact that $d(\mu_{l_j},\mu_{l_i})\leq 4\b$ for all $i,j$, we have that $\n F(t,\mu_{l_i})-F(t,\mu_{l_j})\n\leq 4K\b$ for all $t\in [1,q]$. But as $n>8K$ we get $R(t,x)\neq 0$ for any $t\in [0,1]$. Altogether we obtain $Z([0,q]\times S_{\e})\cap \{0\}=\emptyset$, which was to prove.
\end{proof}
One could ask in general what are the metric spaces $M$ such that $\F(M)$ has an unconditional Schauder basis. It is clear that if $M$ contains a line segment, then $\L_1$ is contained in $\F(M)$ and therefore $\F(M)$ cannot have an unconditional Schauder basis. The only interesting cases are then topologically discrete spaces $M$. Our guess is that if $\F(M)$ has an unconditional Schauder basis, it is isomorphic to $\ell_1$.
\\
\\\textbf{Open problem 1} \textit{Suppose $\F(M)$ has an unconditional Schauder basis. Is it isomorphic to $\ell_1$?}

In \cite{Gd}, one sees that $\F(M)$ is a complemented subspace of $L_1$ if and only if $M$ can be bi-Lipschitzly embedded into an $\R$-tree. A complemented subspace of $\L_1$ with unconditional basis is isomorphic to the space $\ell_1$ due to \cite{LiPe}. One can therefore restate the conjecture above into: Suppose $\F(M)$ has a Schauder basis $B$. If $M$ cannot be embedded into an $\R$-tree, is it true that $B$ is conditional?
\\
\\\textbf{Open problem 2} \textit{Is it true that for every uniformly discrete set $N\subseteq\R^2$ the space $\F(N)$ has a Schauder basis?}

It follows from Corollary \ref{no basis} the answer is no if we restrict ourselves only to retractional Schauder bases. However, we don't know if, supposed the answer is yes, we can find for every uniformly discrete set $N\subseteq\R^2$ an extensional Schauder basis on $\F(N)$.

\end{document}